\documentclass[11pt,reqno]{article}

\usepackage[usenames]{color}
\usepackage{graphicx}
\usepackage{amscd}
\usepackage[english]{babel}
\usepackage{color}
\usepackage{fullpage}
\usepackage{psfig}
\usepackage{graphics,amsmath,amssymb}
\usepackage{amsthm}
\usepackage{amsfonts}
\usepackage{latexsym}
\usepackage{epsf}
\usepackage{float}

\usepackage[colorlinks=true,
linkcolor=webgreen,
filecolor=webbrown,
citecolor=webgreen]{hyperref}

\definecolor{webgreen}{rgb}{0,.5,0}
\definecolor{webbrown}{rgb}{.6,0,0}

\newcommand{\seqnum}[1]{\href{http://oeis.org/#1}{\underline{#1}}}
\newcommand{\beql}[1]{\begin{equation}\label{#1}}
\newcommand{\eeq}{\end{equation}}
\newcommand{\eqn}[1]{(\ref{#1})}

\newcommand{\sL}{{\cal{L}}}

\newcommand{\sR}{{\cal{R}}}

\DeclareMathOperator{\Fib}{Fib}
\DeclareMathOperator{\GF}{GF}
\DeclareMathOperator{\height}{ht}
\DeclareMathOperator{\wt}{wt}
\DeclareMathOperator{\I}{I}
\DeclareMathOperator{\II}{II}
\DeclareMathOperator{\III}{III}
\DeclareMathOperator{\IV}{IV}
\newcommand{\ZZ}{\mathbb Z}

\newtheorem{thm}{Theorem}{\bfseries}{\itshape}
{\bfseries}{\itshape}
{\bfseries}{\itshape}
{\bfseries}{\itshape}
{\bfseries}{\itshape}

\begin{document}
\theoremstyle{plain}

\begin{center}
{\large\bf On the Number of ON Cells in Cellular Automata } \\
\vspace*{+.2in}

N. J. A. Sloane \\ 
The OEIS Foundation Inc., \\
11 South Adelaide Ave., 
Highland Park, NJ 08904, USA \\
Email:  \href{mailto:njasloane@gmail.com}{\tt njasloane@gmail.com} 
\vspace*{+.1in}

March 3, 2015

\vspace*{+.1in}
\textsc{To Ron Graham, commemorating his 80th birthday, \\
and 47 years of friendship}

\vspace*{+.1in}

\begin{abstract}

If a cellular automaton (CA) is started with a single ON cell,
how many cells will be ON after $n$ generations?
For certain ``odd-rule'' CAs, including 
Rule 150, Rule 614, and Fredkin's Replicator,
the answer can be found by using the combination of a new
transformation of sequences, the run length transform,
and some delicate scissor cuts.
Several other CAs are also discussed, although  the 
analysis becomes more difficult as the patterns 
become more intricate.
\end{abstract}
\end{center}


\section{Introduction}\label{Sec1}
When confronted with a number sequence, 
the first thing  is to try to
conjecture a rule or formula, and then 
(the hard part) 
prove that the formula is correct.
This article had its origin
in the study of one such sequence,
$1, 8, 8, 24, 8, 64, 24, 112, 8, 64, 64, 192, \ldots$
(\seqnum{A160239}\footnote{Six-digit numbers prefixed by A
refer to entries in \cite{OEIS}.}), although several similar sequences will 
also be discussed.

These sequences arise from studying how activity 
spreads in cellular automata (for background see
 \cite{Tooth, Epp09, Fred2000, Kari, MOW84,
PaWo85,  Sing03, Ulam62,
Wolf83, Wolf84, NKS}).
 If we start with a single ON cell, how many cells will be ON after $n$ generations?
 The sequence above arises from the CA known as Fredkin's Replicator
 \cite{Layman}.
 In 2014, Hrothgar  sent the author a manuscript \cite{Hroth}
 studying this CA, and  conjectured that the sequence satisfied a certain recurrence.
 One of the goals of the present paper is to prove 
 that this conjecture is correct---see \eqn{EqFred21}.

In Section \ref{SecOdd} we discuss a general class
 (the ``odd-rule'' CAs) to which Fredkin's Replicator belongs,
and in \S\ref{SecRLT} we introduce an operation
on number sequences
(the ``run length transform'')
which helps in understanding the resulting sequences.
Fredkin's Replicator, which is based on the Moore neighborhood,
is the subject of \S\ref{SecFR},
and \S\ref{SecVN} analyzes another odd-rule CA,  based on
the von Neumann neighborhood with a center cell.
Although these two CAs are similar, different
techniques are required for establishing the recurrences.
 Both proofs  involve making scissor cuts to
dissect the configuration of ON cells into recognizable pieces.

Section \ref{SecOther} discusses some other CAs in one,
two, and three dimensions
where it is possible to find a formula, and some for which no formula is presently known.
In dimension one, Stephen Wolfram's 
well-known list \cite{PaWo85, Wolf84, NKS} of 256 different CAs based on a three-celled
neighborhood gives rise to just seven interesting sequences (\S\ref{Sec61}).
Other two-dimensional CAs are discussed
in \S\S\ref{Sec62}, \ref{Sec63}, and
the three-dimensional analog of Fredkin's Replicator
in \S\ref{Sec64}.
The final section (\S\ref{SecFurther}) gives some additional properties of the run length transform.
For many further examples of cellular automata sequences,
see \cite{Tooth} and \cite{OEIS} (the 
index to \cite{OEIS} lists nearly 200 such sequences).


\begin{figure}[!h]
\centerline{\includegraphics[angle=0, width=3.50in]{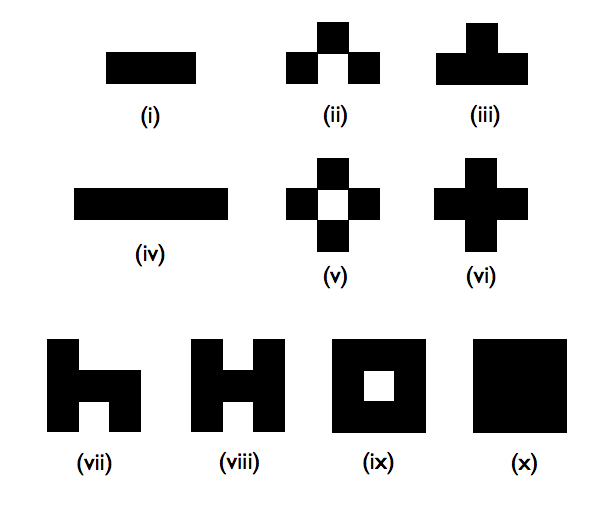}}
\caption{Some two-dimensional neighborhoods (Figs.~(i) and (iv) 
are three and five cells wide, respectively).}
\label{FigXX}
\end{figure}

\section{Odd-rule CAs}\label{SecOdd}

We consider cellular automata  whose cells form a
$d$-dimensional cubic lattice $\ZZ^d$, where $d$ is  
1, 2, or 3.
Each cell is either ON or OFF, and an ON cell
with center at the lattice point
$u = (u_1, u_2, \ldots, u_d) \in \ZZ^d$ 
will be identified with the monomial
$x^u = x_1^{u_1} x_2^{u_2} \cdots x_d^{u_d}$,
which we regard as an element of the ring of Laurent
polynomials $\sR = \GF(2)[x_1, x_1^{-1}, \ldots, x_d, x_d^{-1}]$
with mod 2 coefficients.  
The state of the CA is  specified by giving
the formal sum $S$ of all its ON cells.  As long
as only finitely many cells are ON, $S$ is indeed 
a polynomial in the variables $x_i$ and $x_i^{-1}$,
and is therefore an element of $\sR$.
We write $u \in S$ to indicate that $u$ is ON, i.e., that $x^u$ is
a monomial in $S$.

In most of this paper we will focus on what may be called
``odd-rule'' CAs. An odd-rule CA is defined by specifying 
a neighborhood of the cell at the origin, given by
an element $F \in \sR$ listing the cells
in the neighborhood.
A typical example is the Moore neighborhood in $\ZZ^2$,
which consists of the eight cells surrounding the cell
at the origin in the square grid (see Fig.~\ref{FigXX}(ix)), 
and is specified by
\begin{align}\label{EqFred0}
F  &~:=~ \frac{1}{xy} + \frac{1}{y} + \frac{x}{y}
+ \frac{1}{x}+x
+\frac{y}{x}+y+xy \nonumber \\
& ~=~\left(\frac{1}{x}+1+x\right)\left(\frac{1}{y}+1+y\right) -1 ~  \in  ~ \sR = \GF(2)[x,x^{-1},y,y^{-1}].
\end{align}
The neighborhood of an arbitrary cell $u$ is obtained
by shifting $F$ so it is centered at $u$, that is, by the product
$x^u F \in \sR$. Given $F$, the corresponding
{\em odd-rule} CA is defined by the rule that the cell at $u$ is ON
at generation $n+1$ if it is the neighbor of an odd number of cells that were
ON at generation $n$, and is otherwise OFF.

Our goal is to find $a_n(F)$, the number of ON cells at the $n$th generation
when the CA is started in generation 0 with a single
ON cell at the origin.
For odd-rule CAs there is a simple formula.
The number of nonzero terms in an element $P \in \sR$ will
be denoted by $|P|$.

\begin{thm}\label{Th1}
For an odd-rule CA with neighborhood $F$, 
the state at generation $n$ is equal to $F^n$, and 
$a_n(F) = |F^n|$.
\end{thm}
\begin{proof}
We use induction on $n$.
By definition, the initial state is $1 = F^0$, and $a_0(F) =1$.
The ON cell at the origin turns ON all the cells in $F$,
so the state at generation 1 is $F$ itself, and $a_1(F) = |F|$.
Suppose the state at generation $n$ is $F^n$.
An ON cell $w \in F^n$ will affect a cell $u$ if and only if 
$u$ is in the neighborhood of $w$, that is, if and only if
$u \in wF$.
For $u$ to be turned ON,
there must be an odd number
of cells $w \in F^n$ with $u \in wF$.
Since the coefficients in $\sR$ are evaluated mod 2, 
$u$ will be turned ON if and only if
$u \in \sum_{w \in F^n} wF = F \sum_{w \in F^n}w = F.F^n = F^{n+1}$.
So $F^{n+1}$ is precisely the state at generation $n+1$,
and $a_{n+1}(F) = |F^{n+1}|$.
\end{proof}


\section{The run length transform}\label{SecRLT}
We define an operation on number sequences, 
the ``run length transform''.  For an integer $n \ge 0$,
let $\sL(n)$ denote the list of the lengths of
the maximal runs of 1s in the binary expansion of $n$.
For example, since the binary expansion of 55 is 110111,
containing runs of 1s of lengths 2 and 3,
$\sL(55) = [2,3]$.
$\sL(0)$ is the empty list, and
$\sL(n)$ for $n=1,\ldots,12$ is respectively
$[1], [1], [2], [1], [1,1], [2], [3], [1], [1,1], [1,1], [1,2], [2]$
(\seqnum{A245562}).

\vspace*{+.1in}
\noindent
\textbf{Definition.} The {\em run length transform} of a sequence 
$[S_n, n \ge 0]$ is the sequence $[T_n, n \ge 0]$ 
given by
\beql{EqRLT1}
T_n ~=~ \prod_{i \in \sL(n)} S_i.
\eeq
Note that $T_n$ depends only on the lengths of the runs of 1s 
in the binary expansion of $n$, not on the order
in which they appear. For example, since $\sL(11) = [1,2]$
and $\sL(13)=[2,1]$, $T_{11}=T_{13}= S_1S_2$.
Also $T_0=1$ (the empty product), so the value 
of $S_0$ is never used, and  will usually be taken to be 1.
Further properties and additional examples of the run
length transform will be given in \S\ref{SecFurther}. 
See especially Table \ref{Tab0}, which shows 
how the transformed sequence has a natural division into blocks.
 
Define the {\em height} $\height(P)$ of an element
$P \in \sR$ to be the maximal value of $|e_i|$ in any
monomial $x_1^{e_1} \cdots x_d^{e_d}$ in $P$.
If $\height(P)=h$, the cells in $P$ are contained in a $d$-dimensional 
cube centered at the origin with edges that are $2h+1$ cells long.
Note that $\height(PQ) \le \height(P)+\height(Q)$
and $\height (P^k) = k \height(P)$.

The second property that makes odd-rule CAs
easier to analyze than most  is the following.
\begin{thm}\label{Th2}
If $\height(F) \le 1$, then $[a_n(F), n \ge 0]$ is
the run length transform of the subsequence
\beql{EqRLT2}
[a_0(F), a_1(F), a_3(F), a_7(F),  a_{15}(F),\ldots, a_{2^k-1}(F), \ldots] .
\eeq
\end{thm}
\begin{proof}
The proof depends on the identity sometimes
called the Freshman's Dream, which in its simplest
form states that $(x+y)^2 \equiv x^2+y^2$ mod 2,
and more generally that 
for $P(x_1,x_1^{-1},\ldots) \in \sR$,
\beql{EqFD}
P(x_1,x_1^{-1},\ldots)^{2^k} ~=~ P(x_1^{2^k}, x_1^{-2^k}, \ldots),
\eeq
for any integer $k \ge 0$,
and in particular that
$|P(x_1,x_1^{-1},\ldots)^{2^k} | = |P|$.
Suppose first that the binary expansion of $n$ contains
exactly two runs of 1s, separated by one or more 0s, say
$$
n ~=~ \overbrace{111\cdots1}^{m_1}
\overbrace{00\cdots 0}^{m_2}
\overbrace{111\cdots1}^{m_3},  \quad m_1, m_2, m_3 \ge 1,
$$
i.e.,
$$
n ~=~(2^{m_1}-1)2^{m_2+m_3} + (2^{m_3}-1),
$$
with $\sL(n) = [m_1,m_3]$.
Then
\beql{Eqt1}
F^n ~=~ (F^{2^{m_1}-1})^{2^{m_2+m_3}} F^{2^{m_3}-1}
~=~ P^{2^{m_2+m_3}} Q \mbox{~(say)},
\eeq
where $P:=F^{2^{m_1}-1}$, $Q:=F^{2^{m_3}-1}$.
Equation \eqn{Eqt1} states that $F^n$ 
is a sum of copies of  $Q$
centered at the cells of of $P^{2^{m_2+m_3}}$.
By the Freshman's Dream,
$P^{2^{m_2+m_3}}$ is a polynomial in the variables 
$x_i^{\pm 2^{m_2+m_3}}$,  
so the cells in 
$P^{2^{m_2+m_3}}$ are separated by at least $2^{m_2+m_3}$.
Also, $|P^{2^{m_2+m_3}}|=|P|$.
On the other hand, since $\height(F) \le 1$,
$\height(Q) \le 2^{m_3}-1$,
and since 
$$
2(2^{m_3}-1)+1 ~<~ 2^{m_3+1} ~\le~ 2^{m_2+m_3}
$$ 
the copies of $Q$ in $F^n$ are disjoint from each other, and so
$|F^n| = |P||Q|$, or in other words
$$
a_n(F) ~=~ a_{2^{m_1}-1}(F) ~a_{2^{m_3}-1}(F)
~=~ \prod_{i \in \sL(n)} a_{2^i-1}(F) .
$$
It is straightforward to generalize 
this argument to the case when there are more
than two runs of 1s in 
the binary expansion of $n$, and to establish that for any $n$,
\beql{EqRLT3}
a_n(F) ~=~ \prod_{ i \in \sL(n)} a_{2^i-1}(F),
\eeq
thus completing the proof.
\end{proof}

In several interesting cases the subsequence \eqn{EqRLT2} satisfies 
a three-term linear recurrence,
in which case there is also a simple recurrence for the run length
transform.

\begin{thm}\label{ThRec1}
Suppose the sequence $[S_n, n \ge 0]$ is defined by the recurrence
$S_{n+1} = c_2 S_n + c_3 S_{n-1}, n \ge 1$, with
$S_0=1$, $S_1=c_1$. Then its run length transform $[T_n, n \ge 0]$ 
satisfies the recurrence
\beql{EqRec2}
T_{2t}=T_t,  \quad T_{4t+1} = c_1 T_t, \quad
T_{4t+3} = c_2 T_{2t+1} + c_3 T_{t},
\eeq
for $t>0$, with $T_0=1$.
\end{thm}
\begin{proof}
$T_{2t}=T_t$ is immediate from the definition
of the run length transform, since $\sL(2n) = \sL(n)$.
The binary expansion of $4t+1$ ends in $01$,
so $T_{4t+1}=T_t S_1 = c_1 T_t$.
If $t=2^k-1$ for some $k \ge 1$ then 
$T_{4t+3}=S_{k+2}=c_2 S_{k+1}+c_3 S_k$,
$T_{4t+1}=c_1 S_k$, $T_{2t+1}=S_{k+1}$,
implying
\beql{EqRec3}
T_{4t+3} = c_2 T_{2t+1} + c_3 T_{t}.
\eeq
On the other hand, if $t$ has a zero in its binary expansion,
say $t=i.2^{k+1}+(2^k-1), k \ge 0$,
then 
$T_{4t+3}=T_i S_{k+2}=T_i(c_2 S_{k+1}+c_3 S_k)$,
$T_{4t+1}=T_i c_1 S_k$, $T_{2t+1}=T_i S_{k+1}$,
and again \eqn{EqRec3} follows.
\end{proof}
  

\section{Fredkin's Replicator}\label{SecFR}
The cellular automaton known as {\em Fredkin's Replicator}
\cite{Fred1990, Fred2000, MK06}
is the two-dimensional odd-rule CA defined by the Moore neighborhood $F$
shown in Fig.~\ref{FigXX}(ix) and 
Eq. \eqn{EqFred0}.
(This is the eight-neighbor totalistic Rule 52428 in the Wolfram numbering scheme
\cite{PaWo85, Wolf84, NKS}.)

\begin{figure}[!h]
\centerline{\includegraphics[angle=0, width=4.00in]{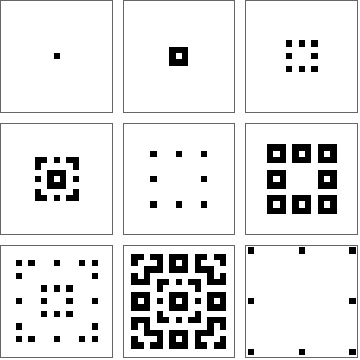}}
\caption{Generations 0 through 8 of the evolution
of Fredkin's Replicator.
ON cells are black, OFF cells are white.}
\label{FigFred1}
\end{figure}

We study the evolution of this CA when it is started at generation 0 
with a single ON cell at
the origin. Generations 0 through 8 are shown in Fig.~\ref{FigFred1}.
 The name of this CA
comes from the fact that any configuration of ON cells will be
replicated eight times at some later stage. For example, generation 1 is
replicated eight times at generation 5.
Although distinctive, the name is not especially appropriate,
since by \eqn{EqFD} any odd-rule CA has a similar replication property.
Let $a_n(F) = a_n$ denote the number of ON cells at the $n$th generation. The initial values of 
$a_n$ are shown in Table \ref{TabFred1}.

\begin{table}[!h]
\caption{ Number of ON cells at $n$th generation of Fredkin's Replicator (\seqnum{A160239}).}
\label{TabFred1}
$$
\begin{array}{c|rrrrrrrrrrrrrrrr}
n& &&&&&& a_n \\
\hline
0& \mathbf{1}  \\
1& \mathbf{8}  \\
2-3& 8 & \mathbf{24}  \\ 
4-7& 8 & 64 & 24 & \mathbf{112}  \\
8-15& 8 & 64 & 64 & 192 & 24 & 192 & 112 & \mathbf{416}  \\ 
16-31& 8 & 64 & 64 & 192 & 64 & 512 & 192 & 896 & 24 & 192 & 192 & 576 & 112 & 896 & 416 & \mathbf{1728}  \\ 
32-63& 8 & 64 & 64 & 192 & 64 & 512 & 192 & 896 & 64 & 512 & 512 & 1536 &
\ldots
\end{array}
$$
\end{table}

Since $\height (F)=1$, we know from Theorem \ref{Th2} that $[a_n, n \ge 0]$ is the
run length transform of the subsequence 
$[b_n = a_{2^n-1}, n \ge 0] = 1, 8, 24, 112, 416, 1728, \ldots $
(shown in bold in Table \ref{TabFred1};
it will turn out to be \seqnum{A246030}).
The main result of this section is the identification of this 
subsequence.

\begin{thm}\label{ThFred1}
The sequence $[b_n, n \ge 0]$ 
 satisfies the recurrence
\beql{EqFred2}
b_{n+1}~=~ 2b_n+8b_{n-1}, ~ {\mbox with~} b_0=1, b_1=8.
\eeq
\end{thm}
\begin{proof}
Let $G_n := F^n$, $H_n := G_{2^n-1} = F^{2^n-1}$.
(Figure \ref{FigFred1} shows $G_0=H_0$, $G_1=H_1$, $G_2$, $G_3=H_2$,
$G_4$, $G_5$, $G_6$, $G_7=H_3$, and $G_8$.)
By definition,
\beql{EqFred3}
H_{n+1} ~=~ F^{2^n} H_n,
\eeq
and, from Theorem \ref{Th1},  $a_n = |G_n|$, $b_n=|H_n|$.

Since $F$ has diameter 3, the nonzero terms $x^i y^j$ in $H_n$ satisfy
\beql{EqFred4}
-(2^n-1) ~\le i,j \le 2^n-1,
\eeq
so we can write
\beql{EqFred5}
H_n ~=~ \sum_{i=-(2^n-1)}^{2^n-1} \,
\sum_{j=-(2^n-1)}^{2^n-1}
H_n(i,j) x^i y^j, 
\eeq
where the coefficient $H_n(i,j) \in \GF(2)$ gives
the state of the cell $(i,j)$ at generation $n$.

From \eqn{EqFred3} and \eqn{EqFD}, $H_{n+1}$ is the sum (in $\sR$)
of eight copies of $H_n$, translated by $2^n$
in each of the N, NW, W, SW, S, SE, E, and NE directions.
That is, for $n \ge 1$,
\begin{align}\label{Eq211.1}
H_{n+1}(i,j) &~=~ 
H_n(i,j-2^n) + H_n(i-2^n,j-2^n) + H_n(i-2^n,j) + H_n(i-2^n,j+2^n)  \nonumber \\
&~~+ H_n(i,j+2^n)  + H_n(i+2^n,j+2^n) + H_n(i+2^n,j) + H_n(i+2^n,j-2^n),
\end{align}
where we adopt the convention that $H_n(i,j)=0$ unless $i$ and $j$
satisfy \eqn{EqFred4}.
Also,\beql{EqFred6}
H_0(0,0)=1, ~~ H_0(i,j)=0 \mbox{~for~} (i,j) \ne (0,0),
\eeq
and $H_1(i,j)=0$ except for
\begin{align}\label{EqFred7}
H_1(0,1)&=H_1(-1,1)=H_1(-1,0)= H_1(-1,-1) \nonumber \\
&=H_1(0,-1)=H_1(1,-1)=H_1(1,0)=H_1(1,1)=1.
\end{align}
By construction, $H_n$ is preserved by the action of the dihedral
group of order 8 (the symmetry group of the square),
generated by the action of $(x,y) \leftrightarrow (y,x)$
and $(x,y) \leftrightarrow (\frac{1}{x},y)$.
We study $H_n$ by breaking it up into the central cell,
the four parts on the axes, and the four quadrants. \\

\noindent
\textbf{The central cell.}
The central cell $H_n(0,0)=1$ if $n=0$,
and (as a consequence of the 8-fold symmetry) is 0 for
$n>0$.
\textbf{The axial parts.}
We define $X_n$ ($n \ge 1$) to be the portion of $H_n$ that lies
on the positive $x$-axis, but normalized so that its center is at the origin:
\beql{EqFred7a}
X_n ~:=~ \frac{1}{x^{2^{n-1}}} \, \sum_{i=1}^{2^n-1} H_n(i,0) x^i.
\eeq
For example, $X_1=1, X_2=\frac{1}{x}+x,
X_3 = \frac{1}{x^3}+\frac{1}{x}+x+x^3$.
From \eqn{Eq211.1} it follows by induction that, for $n \ge 2$,
\beql{EqFred8}
X_n ~=~ \sum_{i=0}^{2^{n-2}-1}\left(x^{-(2i+1)} + x^{2i+1}\right) ~=~
 \left(\frac{1}{x}+x\right)^{2^{n-1}-1}.
\eeq
Similarly,
the portion of $H_n$ that lies
on the negative $x$-axis, normalized so that its center is at the origin, is
\beql{EqFred9}
\widetilde{X}_n ~:=~ x^{2^{n-1}} \, \sum_{i=1}^{2^n-1} H_n(-i,0) x^{-i}
 ~=~  \left(\frac{1}{x}+x\right)^{2^{n-1}-1} ~=~ X_n.
\eeq
Likewise, the normalized portions of $H_n$ on the positive and negative $y$-axes are
\beql{EqFred10}
Y_n ~=~ \widetilde{Y_n} ~=~ \left(\frac{1}{y}+y\right)^{2^{n-1}-1}.
\eeq
\noindent
\textbf{The four quadrants.}
Next, define $\I_n$ for $n\ge 1$ to consist of the portion
of $H_n$ lying in the first quadrant, again
normalized so that its center is at the origin:
\beql{EqFredI}
\I_n ~:=~ \frac{1}{(xy)^{2^{n-1}}} \, \sum_{i=1}^{2^n-1}  \sum_{j=1}^{2^n-1} H_n(i,j) x^i y^j.
\eeq
Similarly, we define
\begin{align}\label{EqFredII}
\II_n &~:=~ \left(\frac{x}{y}\right)^{2^{n-1}} \, \sum_{i=1}^{2^n-1}  \sum_{j=1}^{2^n-1} H_n(-i,j) x^{-i} y^j, \nonumber \\
\III_n &~:=~ (xy)^{2^{n-1}} \, \sum_{i=1}^{2^n-1}  \sum_{j=1}^{2^n-1} H_n(-i,-j) x^{-i} y^{-j}, \nonumber  \\
\IV_n &~:=~ \left(\frac{y}{x}\right)^{2^{n-1}} \, \sum_{i=1}^{2^n-1}  \sum_{j=1}^{2^n-1} H_n(i,-j) x^i y^{-j}.
\end{align}
Assembling the parts, we see that, for $n \ge 1$, $H_n = $
\begin{eqnarray}
 \begin{matrix}
 & ~ & (y/x)^{2^{n-1}} \II_n  & + & y^{2^{n-1}}  Y_n & + & (xy)^{2^{n-1}} \I_n \\
&  + & (1/x)^{2^{n-1}} ~\widetilde{X}_n  & + & 0  & + & ~~x^{2^{n-1}} ~X_n \\
 & + & (xy)^{-2^{n-1}} \III_n  & + &  y^{-2^{n-1}} \widetilde{Y}_n
 & + & (x/y)^{2^{n-1}} \IV_n \nonumber
\end{matrix} 
\end{eqnarray}
which we write as a matrix
\begin{eqnarray}
 H_n ~=~ \begin{bmatrix}
\II_n   &  Y_n  &  \I_n \\
\widetilde{X}_n  & 0  &  X_n \\
\III_n   &  \widetilde{Y}_n &  \IV_n 
\end{bmatrix}, 
\label{EqFred12}
\end{eqnarray}
where it is to be understood that the blocks are to be shifted by the appropriate
amounts (that is, the $\I_n$ in the top right corner
is to be multiplied by $(xy)^{2^{n-1}}$, and so on).
By summing the eight translated copies of $H_n$, as in \eqn{Eq211.1},
we obtain
\begin{eqnarray}
H_{n+1} =  
\begin{bmatrix}
\II_n & Y_n &\I_n+\II_n & Y_n &\I_n+\II_n & Y_n &\I_n \\
\widetilde{X}_n & 0 & X_n+\widetilde{X}_n & 0 & X_n+\widetilde{X}_n & 0 & X_n \\
\II_n+\III_n & Y_n+\widetilde{Y}_n &\I_n+\III_n+\IV_n & \widetilde{Y}_n & \II_n+\III_n+\IV_n & Y_n+\widetilde{Y}_n &\I_n+\IV_n \\
\widetilde{X}_n & 0 & X_n & 0 & \widetilde{X}_n & 0 & X_n \\
\II_n+\III_n & Y_n+\widetilde{Y}_n &\I_n+\II_n+\IV_n & Y_n & \I_n+\II_n+\III_n & Y_n+\widetilde{Y}_n &\I_n+\IV_n \\
\widetilde{X}_n & 0 & X_n+\widetilde{X}_n & 0 & X_n+\widetilde{X}_n & 0 & X_n \\
\III_n & \widetilde{Y}_n & \III_n+\IV_n & \widetilde{Y}_n & \III_n+\IV_n & \widetilde{Y}_n & \IV_n
\end{bmatrix}.
\label{EqFred13}
\end{eqnarray}
Using \eqn{EqFred9} and \eqn{EqFred10}, we have
\begin{eqnarray}
\II_{n+1} =
\begin{bmatrix}
\II_n & Y_n &\I_n+\II_n \\
X_n & 0 & 0 \\
\II_n+\III_n & 0 &\I_n+\III_n+\IV_n 
\end{bmatrix},~~
\I_{n+1} =
\begin{bmatrix}
\I_n+\II_n & Y_n &\I_n \\
 0 & 0 & X_n \\
\II_n+\III_n+\IV_n & 0 &\I_n+\IV_n 
\end{bmatrix},
\label{EqFred14}
\end{eqnarray}
\begin{eqnarray}
\III_{n+1} =
\begin{bmatrix}
\II_n+\III_n & 0 &\I_n+\II_n+\IV_n \\
 X_n & 0 & 0 \\
\III_n & Y_n &\III_n+\IV_n 
\end{bmatrix},~~
\IV_{n+1} =
\begin{bmatrix}
\I_n+\II_n+\III_n & 0 &\I_n+\IV_n \\
0 & 0 & X_n \\
\III_n+\IV_n & Y_n & \IV_n 
\end{bmatrix},
\label{EqFred15}
\end{eqnarray}
By adding these four matrices
we find that $\I_{n+1}+\II_{n+1}+\III_{n+1}+\IV_{n+1}=0$
for $n \ge 1$. This identity is also true for $n=0$,
and we conclude that 
\beql{EqFred16}
\I_n+\II_n+\III_n+\IV_n=0, \quad n \ge 1,
\eeq
and so
\begin{eqnarray}
\II_{n+1} =
\begin{bmatrix}
\II_n & Y_n &\I_n+\II_n \\
X_n & 0 & 0 \\
\II_n+\III_n & 0 &\II_n 
\end{bmatrix},~~
\I_{n+1} =
\begin{bmatrix}
\I_n+\II_n & Y_n &\I_n \\
 0 & 0 & X_n \\
\I_n & 0 &\I_n+\IV_n 
\end{bmatrix}, ~~
\label{EqFred18}
\end{eqnarray}
etc., and finally that, for $n \ge 1$,
\begin{eqnarray}
H_{n+1} =  
\begin{bmatrix}
\II_n & Y_n &\I_n+\II_n & Y_n &\I_n+\II_n & Y_n &\I_n \\
X_n & 0 & 0 & 0 & 0 & 0 & X_n \\
\II_n+\III_n & 0 &\II_n & Y_n & \I_n & 0 &\I_n+\IV_n \\
X_n & 0 & X_n & 0 & X_n & 0 & X_n \\
\II_n+\III_n & 0 & \III_n & Y_n & \IV_n & 0 &\I_n+\IV_n \\
X_n & 0 & 0 & 0 & 0 & 0 & X_n \\
\III_n & Y_n & \III_n+\IV_n & Y_n & \III_n+\IV_n & Y_n & \IV_n
\end{bmatrix}.
\label{EqFred19}
\end{eqnarray}
In \eqn{EqFred19} we see that $H_{n+1}$ contains a copy
of $H_n$ at its center. The four corner blocks together
with two copies each of $X_n$ and $Y_n$ form another,
``deconstructed'', copy of $H_n$.

Suppose $n \ge 2$, and consider the blocks $[\I_n+\II_n ~ Y_n ~ \I_n+\II_n ~  Y_n]$
in the top row of \eqn{EqFred19}.
Using  \eqn{EqFred18} these
blocks can be expanded to give
\begin{eqnarray}
\begin{bmatrix}
\I_{n-1} & 0 & \II_{n-1} & Y_{n-1} & \I_{n-1} & 0 & \II_{n-1} & Y_{n-1} \\
X_{n-1} & 0 & X_{n-1} & 0 & X_{n-1} & 0 & X_{n-1} & 0 \\
\IV_{n-1} & 0 & \III_{n-1} & Y_{n-1} & \IV_{n-1} & 0 & \III_{n-1} & Y_{n-1}
\end{bmatrix}.
\label{EqFred20}
\end{eqnarray}
The central three columns (columns 3, 4, and 5)
 give a copy of $H_{n-1}$, and columns 7, 8, and 1, in that order,
 give another copy.
We get two further copies of $H_{n-1}$ from the
analogous blocks in each of the other three sides of \eqn{EqFred19},
so in total $H_{n+1}$ contains as many ON cells as are in 
two copies of $H_n$ plus eight copies of $H_{n-2}$.
This implies \eqn{EqFred2} for $n \ge 2$.
Equation \eqn{EqFred2} is certainly true for $n=1$, so this completes the proof of the theorem.
\end{proof}

\begin{figure}[!h]
\centerline{\includegraphics[angle=0, width=4.00in]{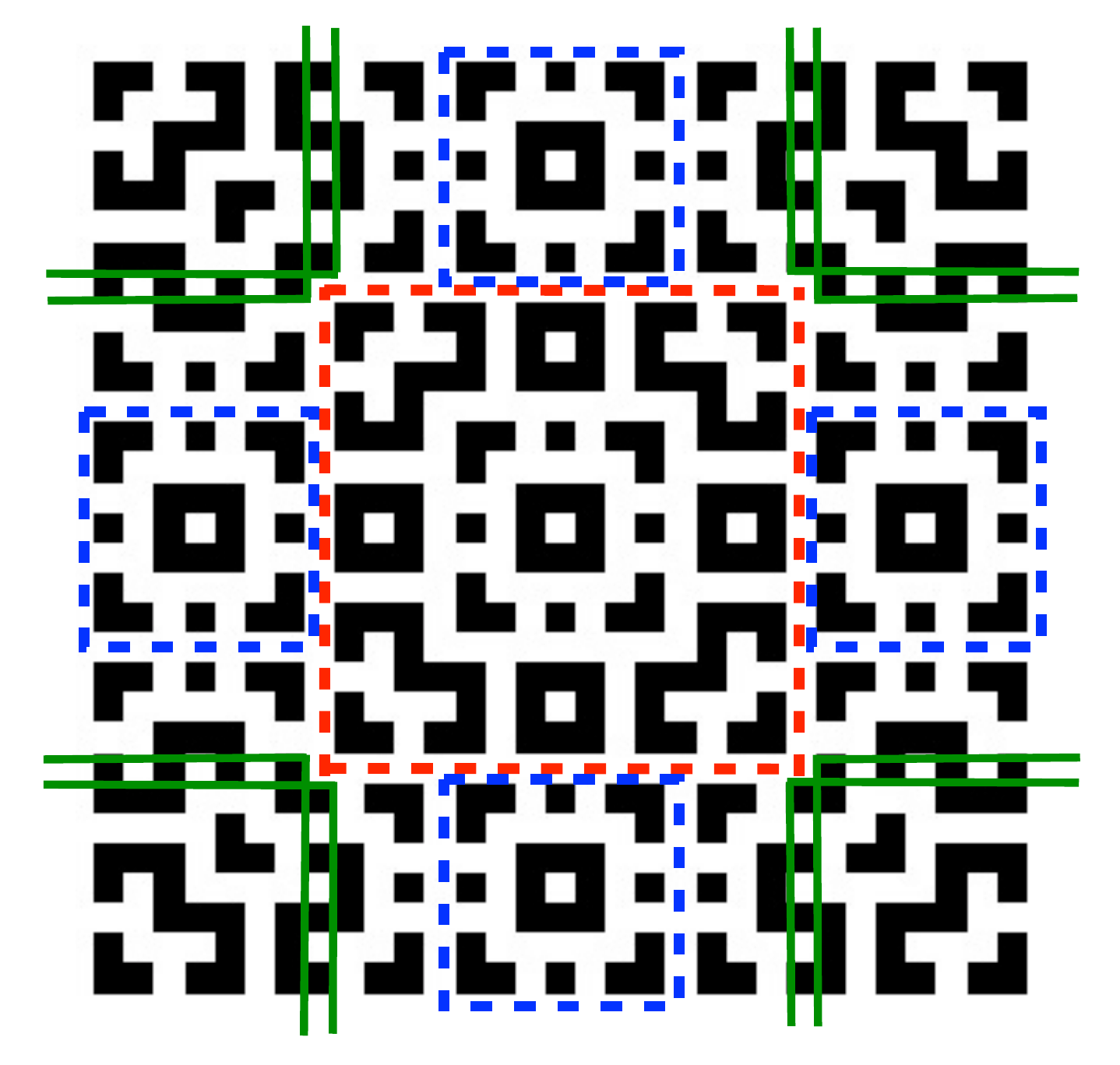}}
\caption{Dissection of $H_4$ into pieces that can be reassembled
to give two copies of $H_3$ and eight copies
of $H_2$, illustrating \eqn{EqFred2} in the case $n=3$.} 
\label{FigFred2}
\end{figure}

The characteristic polynomial of \eqn{EqFred2} is $x^2-2x-8 = (x+2)(x-4)$,
and it follows that
\beql{EqFred1}
b_n ~=~ \frac{5.4^n + (-2)^{n+1}}{3}, ~n \ge 0.
\eeq

In geometric terms, the proof of Theorem \ref{ThFred1}
 shows that $H_{n+1}$  can be dissected into pieces 
that can be reassembled to give two copies of $H_n$ and eight copies of $H_{n-1}$.
Figure \ref{FigFred2} shows this dissection in the case of $H_4$.
The central dashed square (colored red in the online version) encloses a copy of $H_3$.
The  smaller dashed squares on the four sides (colored blue) enclose copies of $H_2$.
The four pairs of vertical parallel lines (colored green) enclose copies
of $Y_3$, and 
the four pairs of horizontal parallel lines (also  green) enclose copies
of $X_3$. The four corners (outside the parallel lines) are,
reading counter-clockwise from the top right corner,
respectively $\I_3$, $\II_3$, $\III_3$, and $\IV_3$, and combine
with two copies each of $X_3$ and $Y_3$ to give the second copy of $H_3$.
Along the top edge, the figure is divided into seven pieces, as in 
the top row of \eqn{EqFred19}. By taking the fifth, sixth, and third
pieces in that order gives another copy of $H_2$,
and three further copies are obtained from the other edges of the figure.

From Theorem \ref{ThRec1} and \eqn{EqFred2}, we see that $a_n$ 
is given by the recurrence 
\beql{EqFred21}
a_{2t}=a_t, \quad a_{4t+1}=8a_t, \quad a_{4t+3}=2a_{2t+1}+8a_t,
\eeq
for $t>0$, with $a_0=1$, as conjectured by Hrothgar \cite{Hroth}.


\section{The centered von Neumann neighborhood} \label{SecVN}

In this section we analyze the two-dimensional 
odd-rule CA defined by the five-celled neighborhood
\beql{EqVN1}
F ~=~ \frac{1}{x}+1+x+\frac{1}{y}+y ~\in~ \sR = \GF(2)[x,x^{-1},y,y^{-1}]
\eeq
shown in Fig.~\ref{FigXX}(vi), 
and consisting of the von Neumann neighborhood together with its center.
(This is the five-neighbor totalistic Rule 614.)
We use the same notation as in the previous section,
except that now $F$ is defined by \eqn{EqVN1} instead of
\eqn{EqFred0}.

The initial values of $[a_n, n \ge 0]$ are
$$
\mathbf{1}, \mathbf{5}, 5, \mathbf{17}, 5, 25, 17, \mathbf{61}, 5, 25, 25, 85, 17, 85, 61, \mathbf{217},\ldots
$$
(\seqnum{A072272}), which we know from Theorem \ref{Th2}
is the run length transform of the subsequence $[b_n, n \ge 0] =  
1, 5, 17, 61, 217, 773, 2753, \ldots$, shown in bold  (\seqnum{A007483}).
Generations 0 through 8 are shown in Fig.~\ref{Fig3}, and Figs.~\ref{Fig5} and
\ref{Fig6} show generation 31.

\begin{figure}[!h]
\centerline{\includegraphics[angle=0, width=4.00in]{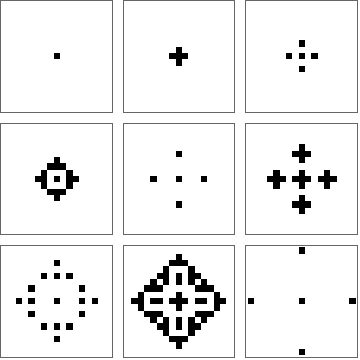}}
\caption{Generations 0 through 8 of the odd-rule CA defined
by the centered von Neumann neighborhood. Generations 0, 1, 3, 7 show
$H_0$, $H_1$, $H_2$, $H_3$ respectively.}
\label{Fig3}
\end{figure}

\begin{figure}[!h]
\centerline{\includegraphics[angle=0, width=4.00in]{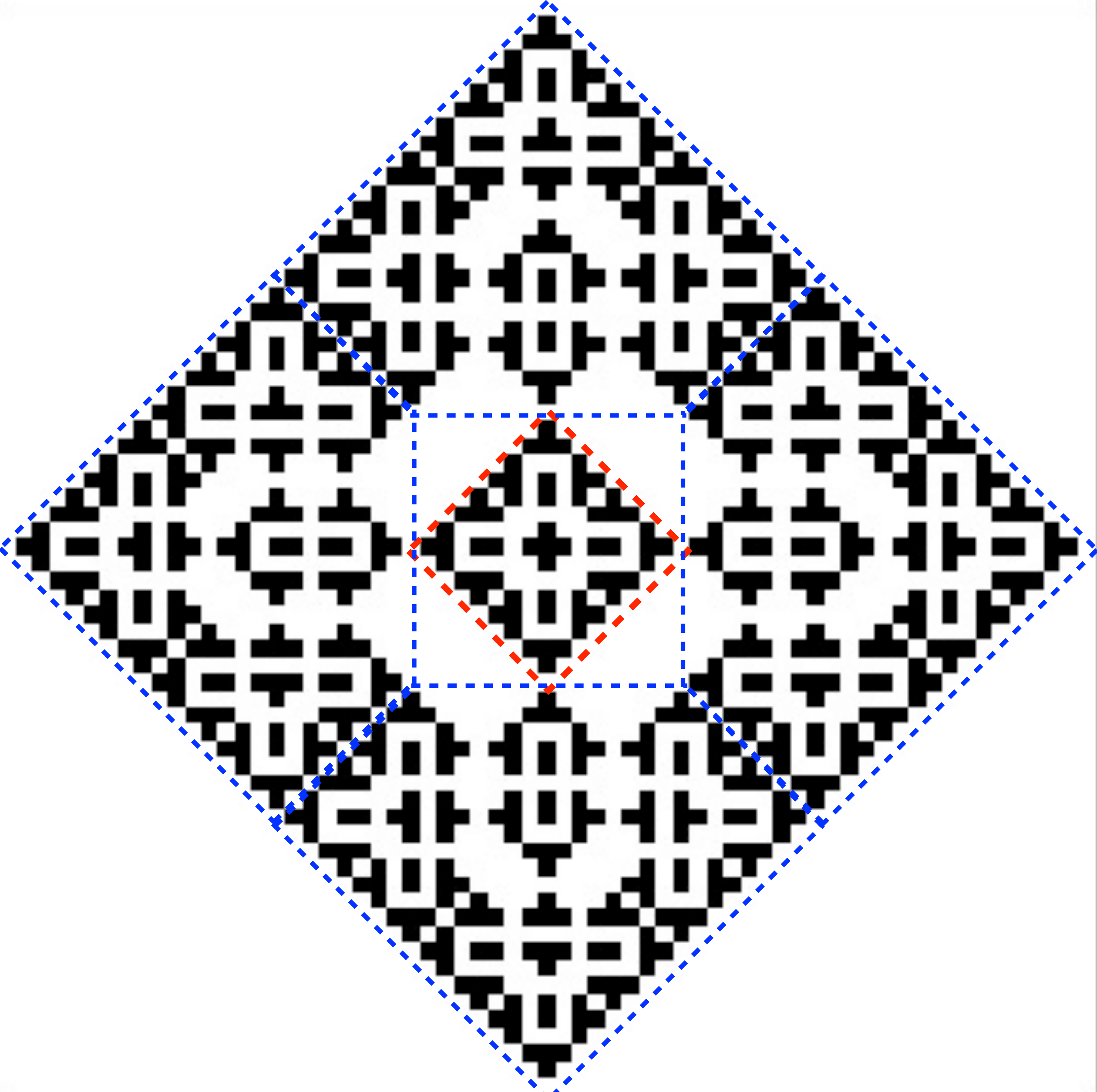}}
\caption{Generation 31 ($H_5$), showing dissection into 
central copy of  $H_3$ (inside red dashed line) and four ``haystacks'' 
$N_5$, $W_5$, $S_5$, $E_5$ (each enclosed by blue dashed line).}
\label{Fig5}
\end{figure}

\begin{figure}[!h]
\centerline{\includegraphics[angle=0, width=4.00in]{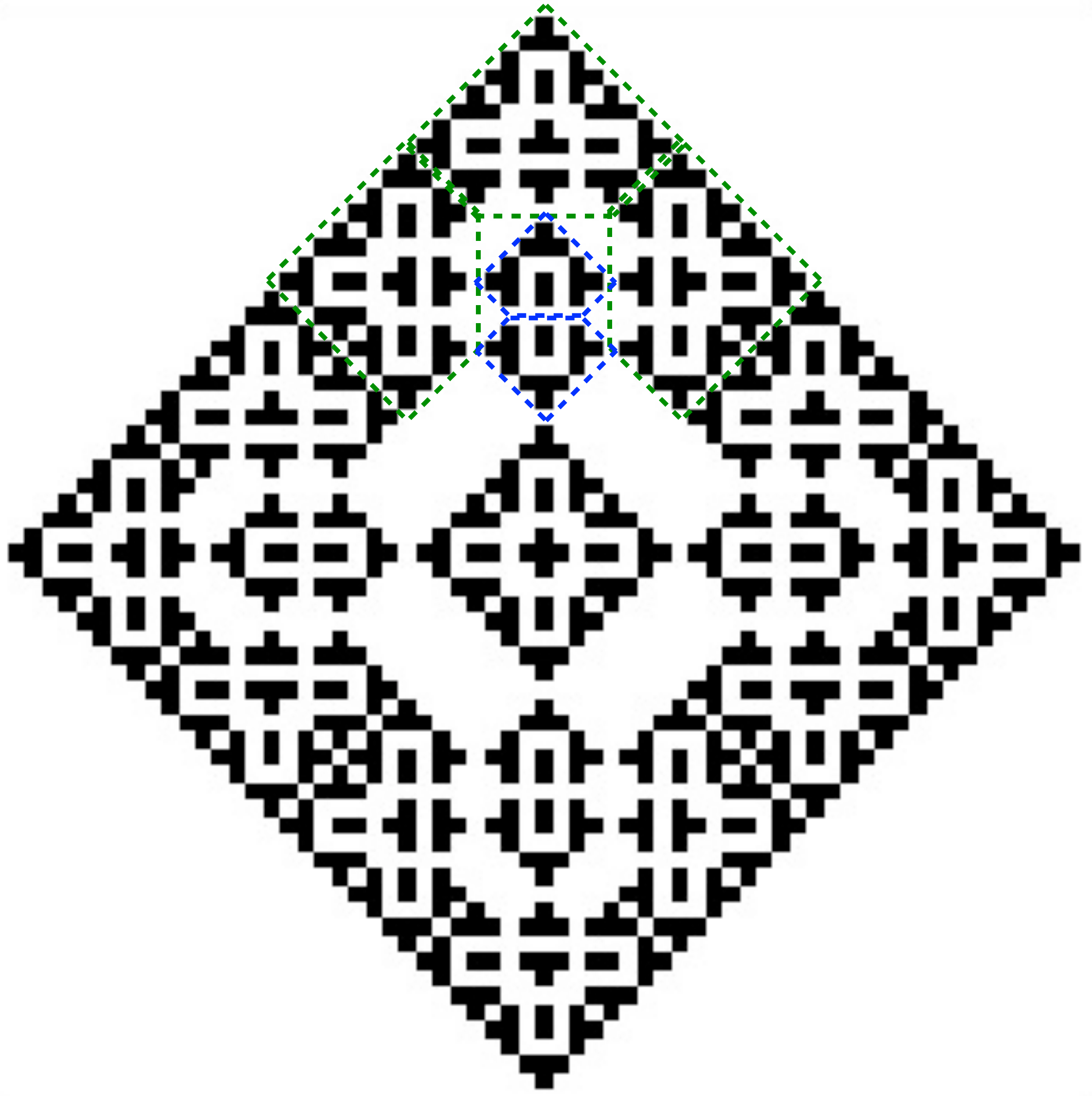}}
\caption{Another view of generation 31 ($H_5$) showing
dissection of haystack $N_5$ into three smaller 
haystacks $N_4$, $W_4$, $E_4$ (each enclosed in 
green dashed line) and two still smaller haystacks $N_3$ and $S_3$
(inside blue dashed lines).}
\label{Fig6}
\end{figure}

\begin{thm}\label{Th5}
The sequence $[b_n, n \ge 0]$ 
 satisfies the recurrence
\beql{EqVN2}
b_{n+1}~=~ 3b_n+2b_{n-1}, ~ \mbox{~with~} b_0=1, b_1=5.
\eeq
\end{thm}
\begin{proof}
As in the proof of Theorem \ref{ThFred1}, $H_n = F^{2^n-1}$,
 $b_n = |H_n|$, and again $H_n$ is  preserved 
under
the action of a  dihedral group of order 8.
The first step in the proof is to show
that  for $n \ge 2$, $H_n$ can be dissected into
a central copy of $H_{n-2}$  and four disjoint 
pentagonal or ``haystack''-shaped regions
(see Fig.~\ref{Fig5} for the dissection of $H_5$).
The four haystacks in $H_n$ will be
denoted by $N_n$, $W_n$, $S_n$, and $E_n$, 
according to the direction in which  they point (a precise definition
will be given below).
They are equivalent under the action
of the dihedral group. Algebraically, we will show that
\beql{EqVN3}
H_n ~=~ H_{n-2} + y^{2^{n-1}} N_n 
+ x^{-2^{n-1}} W_n +  y^{-2^{n-1}} S_n + x^{2^{n-1}} E_n,
\eeq 
where the five polynomials on the right are disjoint 
(i.e., have no monomials in common with each other).

Once we have established \eqn{EqVN3}, the second
step in the proof will be to show that each haystack can be dissected into
five smaller haystacks (see Fig.~\ref{Fig6} for the dissection of $N_5$).
In particular, we will show that for $n \ge 3$,
\beql{EqVN4}
N_{n} ~=~ 
y^{2^{n-2}} N_{n-1}+x^{-2^{n-2}} W_{n-1}+x^{2^{n-2}} E_{n-1}
+N_{n-2}+y^{-2^{n-3}} S_{n-2},
\eeq
where again the polynomials on
the right are disjoint.
Let $\nu_n := |N_n| = |W_n| = |S_n| = |E_n|$. Then
\eqn{EqVN4} implies $\nu_n = 3\nu _{n-1}+2 \nu _{n-2}$.
From \eqn{EqVN3} we have
$b_n =b_{n-2}+4\nu _n$,
so
$$
b_{n+1} -3b_{n}-2b_{n-1} = b_{n-1}-3b_{ n-2} -2b_{n-3}
= \cdots = \mbox{~either~} b_3-3b_2-2b_1 
 \mbox{~or~} b_2-3b_1-2b_0,
 $$
 and each of the last two expressions evaluates to zero.
 This will complete the proof of \eqn{EqVN2}.
 It is worth remarking that these dissections 
 are of a different nature from the dissection in the previous
 section. There it was necessary to make some
 non-obvious cuts through the contiguous blocks
 of ON cells, as shown by the parallel (green) lines
 in the corners of Fig.~\ref{FigFred2}. 
 In contrast, in the present proof, the dissections are carried out by 
 ``tearing'' along the obvious  ``perforations'', rather like tearing apart a block
 of postage stamps.

Now to the details.  It follows
from the definition (see the sequence
of successive states in Fig.~\ref{Fig3}) that
$H_n$ is a diamond-shaped configuration
with extreme points $(\pm(2^n-1),0)$, $(0,\pm(2^n-1))$.
Also,  for $n \ge 2$, $H_n$ contains a copy of $H_{n-2}$ at its center,
surrounded by a layer, at least one cell wide,
of OFF cells. This follows from the identity
$$
H_n - H_{n-2} ~=~ F^{2^n-1}-F^{2^{n-2}-1} 
~=~ H_{n-2}(1+H_2^{2^{n-2}}),
$$
upon checking that the right-hand side
contains no monomials $x^i y^j$ with $|i|+|j| \le 2^{n-2}$.
The buffer layer of OFF cells around the central $H_{n-2}$
consists of the cells $x^i y^j$ with
$|i|+|j| = 2^{n-2}$.

We define the $n$th North haystack to be
\beql{EqVN5}
N_n ~:=~ H_{n-2}\, T_N^{2^{n-2}}, ~n \ge 2,
\eeq
where $T_N := 1/x+1+x+y$ is the four-celled North-pointing triangle
shown in Fig.~\ref{FigXX}(iii). Similarly, the West, South, and East haystacks are
\beql{EqVN6}
W_n ~:=~ H_{n-2}\, T_W^{2^{n-2}},~
S_n ~:=~ H_{n-2}\, T_S^{2^{n-2}},~
E_n ~:=~ H_{n-2}\, T_E^{2^{n-2}},
\eeq
where $T_W := 1/x +1/y+1+y$ is a West-pointing version 
of $T_N$, and similarly for $T_S$ and $T_E$.
(The simple expressions in \eqn{EqVN5} and \eqn{EqVN6}
were guessed by computing the actual haystacks in $H_3$ to $H_6$ and
using Maple to factor them in $\sR$.)
The haystack $N_n$ has the property that all its cells are on or inside
the convex hull of the five cells $x^iy^j$ with $(i,j)$ equal to
\beql{EqVN7}
(0,2^{n-1}-1), (-2^{n-1}+1,0), 
 (-2^{n-2},-2^{n-2}+1), (2^{n-2}-2^{n-2}+1),
  (-2^{n-1}+1,0).
 \eeq
 To see this, consider what happens to $N_n$ one
 generation later: it becomes 
 \beql{EqVN7a}
 F N_n ~=~ (F T_N)^{2^{n-2}},
 \eeq
using the Freshman's Dream \eqn{EqFD}, where $FT_n$ 
is the six-celled configuration 
$$
y^2+\frac{1}{x^2}+\frac{1}{xy} + \frac{1}{y} + \frac{x}{y}+x^2.
$$
From \eqn{EqVN7a},  $F N_n$ is this configuration with the cells moved
$2^{n-2}$ steps apart, and the cells \eqn{EqVN7} 
lie just inside it. 

Note that the point $(0,0)$ is in the interior
of $N_n$ at the intersection of the vertical line through
the apex and the horizontal line joining the most Western and Eastern points.
The powers of $x$ and $y$ in \eqn{EqVN3} and \eqn{EqVN4}
are needed in order to translate the haystacks
into their correct positions. Since we now know the boundaries of 
all the terms on the right side of
\eqn{EqVN3}, we can check that these five polynomials are indeed disjoint.
We must still check that \eqn{EqVN3} is an identity.
Using the Freshman's Dream, this reduces to 
checking the identity
$$
H_2 ~=~ 1 + y^2 T_N + x^{-2} T_W+y^{-2} T_S+x^2 T_E,
$$
which is true. This completes the proof that \eqn{EqVN3} is a 
proper dissection of $H_n$. The correctness of the dissection \eqn{EqVN4}
is verified in a similar way; we omit the details.
\end{proof}


\section{Other Cellular Automata}\label{SecOther}

\subsection{The 256 one-dimensional rules}\label{Sec61}
There are 256 possible CAs based on the
one-dimensional three-celled  neighborhood
shown in Fig.~\ref{FigXX}(i). 
These are the CAs labeled Rule 0 through Rule 255
in the Wolfram numbering scheme
\cite{PaWo85, Wolf84, NKS}. 
As usual we assume
the automaton is started with a single ON cell, and let
$a_n$ denote the number of ON
cells after $n$ generations.  Illustrations of
the initial generations of all 256 CAs are 
shown on pages 54--56 of \cite{NKS}.
Many of these sequences were analyzed in \cite{Wolf83}; see also \cite{Weiss30}.
If we eliminate those in which some $a_n$ is infinite,
or the sequence $[a_n, n \ge 0]$ is trivial (essentially linear),
or is a duplicate of
one of the others, we are left with just seven sequences:
\begin{list}{\textbullet}{\setlength{\itemsep}{0.00in}}
\item
Rule 18 (or Rule 90; Rule 182 is very similar): $a_n = 2^{\wt(n)}$ (Gould's
sequence \seqnum{A001316}), where $\wt(n)$ is the number of 1s in the
binary expansion of $n$. This is the run length transform of the powers of 2.
\item
Rule 22:  $a_n = 2^{\wt(n)}$ if $n$ even, $3.2^{\wt(n)-1}$ if $n$ odd (\seqnum{A071044}).
\item
Rule 30: The behavior appears chaotic, 
even when started with a single ON cell \cite{Weiss30, 
NKS}.  The sequence, \seqnum{A070952}, has roughly linear growth, but
it seems likely that there is no simpler way to obtain it than by its 
definition.
\item
Rule 62: $a_{n+7}=a_{n+4}+a_{n+3}-a_n$ 
with initial terms $1,3,3,6,5,8,9$ (\seqnum{A071047}).
\item
Rule 110: Although the initial  behavior is chaotic, 
it is an astonishing fact, pointed out by Wolfram \cite[p.~39]{NKS},
that after about three thousand terms all the irregularities disappear.
By using the Salvy-Zimmermann  \texttt{gfun} package in Maple \cite{GFUN}, we
find that the sequence, \seqnum{A071049}, satisfies a linear recurrence
of order 469: for $n \ge 2854$,
\beql{eqGFUN}
a_{n + 469} ~=~ -a_{n + 453}+a_{n + 256}+ a_{n + 240}+a_{n + 229}+a_{n + 213}
-a_{n + 16}-a_n.
\eeq 
This recurrence is far nicer than it initially appears: the
coefficients are palindromic, and its characteristic polynomial 
is the product of 25 irreducible factors.
\item
Rule 126: $a_n= 2^{\wt(n)+1}$, except we must subtract 1 if $n=2^k-1$ for some $k$ (\seqnum{A071051}).
\item
Rule 150: This is the odd-rule CA defined by the three-celled neighborhood. 
The sequence $[a_n, n \ge 0]$ (\seqnum{A071053})
was analyzed by Wolfram \cite{Wolf83}
 (see also \cite{HH02}, \cite{Sillke}). 
 In the notation of the present paper $[a_n, n \ge 0]$ is the run length transform of the 
 Jacobstahl sequence \seqnum{A001045}.
 Theorems \ref{Th1} and  \ref{Th2} 
 were suggested by reading  Sillke's  analysis \cite{Sillke}.
\end{list}

\subsection{Other odd-rule CAs}\label{Sec62}
In this section we discuss  the odd-rule CAs
defined by the height-one neighborhoods  in Fig.~\ref{FigXX}.
(The height-two neighborhood of Fig.~\ref{FigXX}(iv) is
discussed in the last section of the paper.)
Table \ref{Tab62} summarizes the results.
The first column specifies the neighborhood $F$ in Fig.~\ref{FigXX}, 
$a_n(F)$ is the number of ON cells
at generation $n$, $b_n(F)$ denotes
the sequence of which $a_n(F)$ is the run length transform, and
the fourth column gives a  generating function (g.f.) for  
$b_n(f)$. 
For Fig.~(iii), $\Fib_{n+2}$ denotes a Fibonacci number.
The g.f. for (vii), found by Doron Zeilberger \cite{ESZ}, is
\beql{Eqviigf}
\frac { \left(1 + 2\,x \right)
\left(1 + x - x^2 + x^3 + 2\,{x}^{5} \right) }
{1 - 3\,x - 3\,{x}^{2}+{x}^{3}+6\,x^4 -10\,x^5 +8\,{x}^{6}-8\,{x}^{7}} \,.
\eeq
An expanded version of this table, analyzing all the
sequences arising from odd-rule CAs 
defined by height-one neighborhoods on the square grid,
will be published elsewhere \cite{3by3}.

\renewcommand{\arraystretch}{1.2}
 \begin{table}[!h]
\caption{ Odd-rule CAs defined by height-1 neighborhoods $F$ in Fig.~\ref{FigXX}.}
\label{Tab62}
$$
\begin{array}{|c|c|c|c|c|}
\hline
F & a_n(F) & b_n(F) &  \mbox{g.f.} & \mbox{Notes} \\
\hline
\mbox{(i)} & \seqnum{A071053} & \seqnum{A001045} & \frac{1+2x}{1-x-2x^2} & \mbox{Rule~150},~ \S\ref{Sec61}  \\
\mbox{(ii)} & \seqnum{A048883} & \seqnum{A000244} & \frac{1}{1-3x} & a_n = 3^{\wt(n)}  \\
\mbox{(iii)} & \seqnum{A253064} & \seqnum{A087206} & \frac{1+2x}{1-2x-4x^2} & b_n = 2^n\Fib _{n+2}  \\
\mbox{(v)} & \seqnum{A102376} & \seqnum{A000302} & \frac{1}{1-4x} & a_n = 4^{\wt(n)}  \\
\mbox{(vi)} & \seqnum{A072272} & \seqnum{A007483} & \frac{1+2x}{1-3x-2x^2} & \S\ref{SecVN}  \\
\mbox{(vii)} & \seqnum{A253069} & \seqnum{A253070} & \eqn{Eqviigf}
 &    \\
\mbox{(viii)} & \seqnum{A246039} & \seqnum{A246038} & \frac{(1+2x)(1+2x+4x^2)}{1-3x-8x^3-8x^4} & \\
\mbox{(ix)} & \seqnum{A160239} & \seqnum{A246030} & \frac{1+6x}{1-2x-8x^2} & \mbox{Fredkin~Replicator}, ~\S\ref{SecFR}  \\
\mbox{(x)} & \seqnum{A246035} & \seqnum{A139818} & \frac{1+6x-8x^2}{(1-x)(1+2x)(1-4x)} & \mbox{Squares~of~entries~from~(i)} \\
\hline
\end{array}
$$
\end{table}
\renewcommand{\arraystretch}{1.0}

\subsection{Further two-dimensional CAs}\label{Sec63}
If we drop the ``odd-rule'' definition, the number of CAs grows
astronomically---there are
$2^{512}$ based on the
Moore neighborhood alone. 
Pages 171--175 of \cite{NKS}
show many examples of the subset of ``totalistic'' rules,
in which the next state of a cell depends only on its present
state and  the total
number of ON cells surrounding it.
All of these are potential sources of sequences.
In a few cases it is possible to analyze the  sequence,
but usually it seems that no formula or recurrence exists.
In this section we give three examples: one that can be analyzed,
one that might be analyzable with further research,
and one (typical of the majority) where the 
state diagrams are aesthetically appealing but  finding
a formula seems hopeless. 
All three are  totalistic rules, the first two being
based on the von Neumann neighborhood (Fig.~\ref{FigXX}(v)),
and the third on the Moore neighborhood (Fig.~\ref{FigXX}(ix)).

The first example  is  the Rule 750 automaton,
in which an OFF cell turns ON if an odd number of its 
four neighbors are ON, and once a cell is ON it stays ON
\cite[p.~925]{NKS}.
This CA is a hybrid of the ``odd-rule'' CAs studied above
and the ``once a cell is ON it stays ON'' rules studied in \cite{Tooth}.
Here it is convenient to call the initial ON cell generation 1 (rather than 0).
The numbers of ON cells in the first few generations (\seqnum{A169707})  are given in
Table \ref{Tab750}.
 
\begin{table}[!h]
\caption{ Number of ON cells at $n$th generation of CA defined by Rule 750.}
\label{Tab750}
$$
\begin{array}{c|rrrrrrrrr}
n& && a_n \\
\hline
1& 1  \\
2-3& 5 & 9  \\
4-7& 21 & 25 & 37 & 57  \\
8-15& 85 & 89 & 101 & 121 & 149 & 169 & 213 & 281  \\
16-31& 341 & 345 & 357 & 377& 405 & 425 & 469 &  \ldots
\end{array}
$$
\end{table}

The evolution of this CA is similar to several that were studied
in \cite{Tooth}: at generation $2^k$, for $k \ge 2$, 
the structure is enclosed in a diamond-shaped region, 
which is saturated in the sense that no additional interior cells can ever be turned ON,
and contains $a_{2^k}= (4^{k+1}-1)/3$ ON cells.
Then in generations $2^k+1$ to $2^{k+1}-1$,
the structure grows outwards from the four vertices of the diamond,
and the first half of the growth that follows generation $2^{k}$ is  the same as the 
the growth that followed generation $2^{k-1}$.
Figure~\ref{Fig750} shows generation $20=2^4+4$,
where we can see that 16 cells have grown out of each vertex.
Pictures of generations $16=2^3+4$ and $36=2^5+4$  show exactly
the same growth from the vertices (although with
different numbers of ON cells in  the central diamond). 
 
\begin{figure}[!h]
\centerline{\includegraphics[angle=0, width=3.00in]{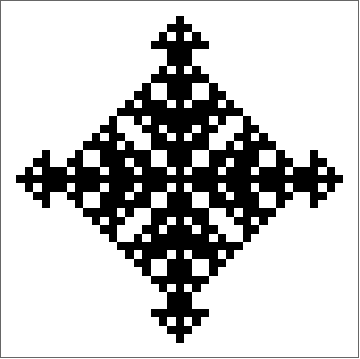}}
\caption{Generation 20 of CA defined by Rule 750, showing $341+4(1+3+5+7)=405$ ON cells, illustrating \eqn{Eq752}.}
\label{Fig750}
\end{figure}

 The successive numbers of ON cells added to a vertex 
 in the generations from  $2^k$ to
 $2^{k+1}-1$ are  $0, 1, 3, 5, 7, 5, 11, 17, 15, 5,\ldots$,
 which form the initial terms $v_0, v_1, v_2, \ldots$  of a sequence (\seqnum{A151548})
 encountered in \cite{Tooth}. The $v_i$ have generating function
 \beql{Eq751}
 \frac{x}{1+x} ~+~ 4 \, x^2  \, \prod_{r=1}^{\infty} (1+x^{2^r-1}+2 x^{2^r}).
 \eeq
Then, for $k\ge 0$ and $0 \le m <2^k$, we have
\beql{Eq752}
 a_{2^k+m} ~=~ \frac{4^{k+1}-1}{3}
 ~+~ 4 \sum_{i=0}^{m} v_i.
\eeq
Bearing in mind the warning in the first sentence
of this paper, we must admit that we have not written
out a complete proof that \eqn{Eq752} is correct. 
However, there should be no difficulty in filling in the details: as the automaton evolves
from generation $2^k$ to $2^{k+1}$, the structure has a natural
dissection into polygonal pieces.

The second example is more speculative: this is
 the  Rule 493 automaton
\cite[p.~173]{NKS}, \seqnum{A246333}.
The binary expansions of 493 and 750 differ
in just four places, so it is not surprising
that this is similar to the previous example.
Now an ON cell stays ON unless exactly zero or four of its neighbors
are ON, in which case it turns OFF, and an
OFF cell turns ON unless exactly two of its neighbors are ON.
Assuming here that we start with a single ON
cell at generation 0,  in the even-numbered generations
the number of ON cells is finite (\seqnum{A246334}): 
$$
1, 5, 17, 29, 61, 73, 109, 157, 229, 241, 277, 329, 429, 477, 573, 633, 861,  \ldots,
$$
while in the odd-numbered generations the number of OFF cells is finite (\seqnum{A246335}):
$$
1, 5, 9, 21, 25, 37, 57, 85, 89, 101, 121, 165, 169, 213, 217, 317, 321, 333, \ldots .
$$
The reason for hoping this automaton might be analyzable is that
the latter sequence agrees with the sequence in Table \ref{Tab750}
up though the eleventh term, 121, after which the sequences
diverge.
Even the respective states are the same up through the sixth term, 37,
although to see this one has to work with the negatives---in
the photographer's sense, interchanging black and while cells---and then
rotating the result by 45 degrees. This needs
further investigation.

\begin{figure}[!h]
\centerline{\includegraphics[angle=0, width=3.00in]{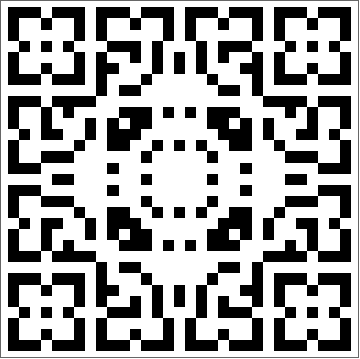}}
\caption{Generation 15 of CA defined by eight-neighbor Rule 780.}
\label{Fig780}
\end{figure}

The third example in the eight-neighbor Rule 780 (\seqnum{A246310}),
in which a cell turns ON if one or four of its neighbors in ON, and
otherwise turns OFF. Although the initial generations
are simple enough, already by generation 15 (Fig.~\ref{Fig780})
the structure is extremely complicated.  Is there a recurrence? 
Is the five-neighbor analog (\seqnum{A253086}) any easier to understand?

\subsection{The three-dimensional analog of Fredkin's Replicator}\label{Sec64}
The three-dimensional Moore neighborhood,
that is, the $3 \times 3 \times 3$ cube without its center cell,
gives rise to the sequence
$1, 26, 26, 124, 26, 676, 124, 1400, \ldots$ (\seqnum{A246031}),
which by Theorem \ref{Th2} is the run length
transform of the subsequence 
\beql{EqMM}
1, 26, 124, 1400, 10000, 89504, 707008, 5924480, 47900416, 393069824, 3189761536, 25963397888,\ldots
\eeq
(\seqnum{A246032}),
computed by Roman Pearce and Michael Monagan.
Doron Zeilberger \cite{ESZ} has found a generating function, a
rational function with numerator of degree $10$ and
denominator of degree $11$, as well as a proof that it is correct.

\section{Further remarks about run length transforms}\label{SecFurther}
\textbf{Block structure.}
It is a surprising fact that the growth sequences
of many CAs have a natural division into blocks of successive lengths 
$(1,)1,2,4,8,16,32,\ldots$. This is true even for some CAs
that are defined on lattices other than  $\ZZ^d$  \cite{Tooth}.
Some of these examples are explained by
the fact that the run length transform 
 always has this property---the division
 into blocks of the  run length transform $[T_n, n \ge 0]$ of an arbitrary 
 sequence $S = [1,A,B,C,D, \ldots]$
is shown in Table \ref{Tab0}.
Table \ref{TabFred1} above gives a concrete example.
The first half of each row is given by $A$ times the beginning of
the $[T_n]$ sequence itself.

\begin{table}[!h]
\caption{ The run length transform $[T_n, n \ge 0]$ of a sequence $S = [1,A,B,C,D, \ldots]$,
showing the division into blocks of sizes 
$1,1,2,4,8,16,\ldots$.}\label{Tab0}
$$
\begin{array}{c|rrrrrrrrrrrrrrrr}
n& &&&&&& T_n \\
\hline
0& 1  \\
1& A \\
2-3& A & B  \\
4-7& A & A^2 & B & C  \\
8-15& A & A^2 & A^2 & AB & B & AB & C & D  \\
16-31& A & A^2 & A^2 & AB & A^2 & A^3 & AB & AC & B & AB & AB & B^2 & C & AC & D & E  \\
32-63& A &  A^2 & A^2 & AB & A^2  & A^3 & AB   & AC  & \ldots   &   &   &   &
\end{array}
$$
\end{table}

\textbf{Further examples.}
We briefly mention four  additional examples of run length transforms.
The run length transform of  $0,1,2,3,4,\ldots$
is $1, 1, 1, 2, 1, 1, 2, 3, 1,\ldots$  (\seqnum{A227349}), which gives the product of the
lengths of runs of 1s in the binary representation of $n$. 
The primes, prefixed by 1, give $1, 2, 2, 3, 2, 4, 3, 5, 2, \ldots$ (\seqnum{A246029}). 
The squares give $1, 1, 1, 4, 1, 1, 4, 9,\ldots$ (\seqnum{A246595}).
The powers of 2 give 
$1, 2, 2, 4, 2, 4, 4, 8, 2,\ldots$,
 $2^{\wt(n)}$
 (\seqnum{A001316}, already mentioned in \S\ref{Sec61}).
 
 \textbf{Graphs.}
 The graphs of run length transforms are usually highly irregular, as one expects from Table 
 \ref{Tab0}. The partial sums of these
 sequences are naturally smoother, and generally have a family resemblance,
 with a bumpy appearance somewhat similar to
 what is seen  in the Tagaki curve \cite{Tagaki11}, \cite{Lag2011}.
 The partial sums of the
 four examples in the previous paragraph
 are \seqnum{A253083}, \seqnum{A253081}, \seqnum{A253082}, \seqnum{A006046},
 respectively, and the partial sums of the sequences arising
 from Fredkin's Replicator, the sequence in Sect. \ref{SecVN}, and the Rule 150
 sequence in Sect. \ref{SecOther}  are respectively
 \seqnum{A245542}, \seqnum{A253908}, \seqnum{A134659}.\footnote{The ``graph''
 button in \cite{OEIS} makes it easy to compare these graphs. However, it is not clear
 how the growth rate of the original sequence affects the
  ``bumpiness'' of the partial sums.}
The latter sequence is discussed in \cite{FSS}, and it
 would be interesting to see
if the methods of that paper
 can be applied to the other six sequences.
 Also, is there any direct connection between the limiting form of 
 these graphs for large $n$ and the Tagaki curve?
 
 \textbf{The generalized run length transform.}
 There are analogs of Theorem \ref{Th2}
 which apply to larger neighborhoods, although they are more complicated
 and not as useful. 
 The following is a version
 which applies when the neighborhood $F$ has height
 at most 2.
 Whereas in Theorem \ref{Th2}, $a_n(F)$ was expressed as a product
 of terms from the subsequence $a_m(F)$ where the binary expansion
 of $m$ contained no zeros, now we
  need the values $a_m(F)$ where $m$
 is any number whose binary expansion
 begins and ends with 1 and 
 does not contain any pair of adjacent zeros. 
 These are the numbers (\seqnum{A247648}) 
 \beql{EqGRLT}
 1, 3, 5, 7, 11, 13, 15, 21, 23, 27, 29, 31, 43, 45, 47, 53, 55, 59, 61, 63, \ldots .
 \eeq
 Suppose for simplicity that the binary expansion of $n$ has the form
$$
n ~=~ \overbrace{\ast \ast \cdots \ast}^{m_1}
\overbrace{00\cdots 0}^{m_2}
\overbrace{\ast \ast \cdots \ast }^{m_3},  \quad m_1, m_3 \ge 1, m_2 \ge 2,
$$
where the asterisks indicate strings of 0s and 1s that
begin and end with 1s and  do not contain any pair of adjacent zeros. 
 If the first such string represents $N_1$ and the second $N_2$, 
 then $a_n(F) = a_{N_1}(F) a_{N_2}(F)$.
 There is an analogous expression in the general case, 
 expressing $a_n(F)$ as a product of  terms
 $a_m(F)$ where $m$ belongs to \eqn{EqGRLT}.

To illustrate, suppose $F$ is the five-celled one-dimensional
neighborhood shown in Fig.~\ref{FigXX}(iv), with height 2. 
The initial values of $a_n(F)$ are
given in Table \ref{TabGRLT}, with $a_m(F)$ for $m$ in \eqn{EqGRLT}
shown in bold. For example, the binary expansion of 
167 is 10100111, so the generalized run
length transform tells us that $a_{167}(F) = a_5(F)a_7(F) = 17.19 = 323$.
It follows from the generalized run length transform property
that in each row of the table, the first one-eighth
of the terms coincide with 5 times the beginning of the sequence itself. 
 
\begin{table}[!h]
\caption{ Number of ON cells at $n$th generation of of odd-rule 
one-dimensional CA defined
by a 5-celled neighborhood (\seqnum{A247649}).}
\label{TabGRLT}
$$
\begin{array}{c|rrrrrrrrr}
n& && a_n \\
\hline
0& \mathbf{1}  \\
1& \mathbf{5}  \\
2-3& 5 & \mathbf{7}  \\
4-7& 5 & \mathbf{17} & 7 & \mathbf{19}  \\
8-15& 5 & 25 & 17 & \mathbf{19} & 7 & \mathbf{31} & 19 & \mathbf{25}  \\
16-31& 5 & 25 & 25 & 35 & 17 & \mathbf{61} & 19 & \mathbf{71} & \ldots 
\end{array}
$$
\end{table}
The bold-faced entries in Table~\ref{TabGRLT} form \seqnum{A253085}, and we end with
one last question:
is there an independent characterization of this sequence?
An affirmative answer might make the
generalized run length transform a lot more interesting.
Much remains to be done in this subject! 


\section*{Postscript, March 2015}
After seeing an initial version of this paper,
Doron Zeilberger observed that it is
possible to use Theorems \ref{Th1} and \ref{Th2} to automate 
 calculation of sequences giving
the number of ON cells in odd-rule CAs, and in
the case of height-one neighborhoods,
to find and rigorously prove the correctness of
generating functions
for the sequences of which they are the run
length transforms. Details will appear elsewhere \cite{ESZ, 3by3}.


\section*{Acknowledgments}
Theorems \ref{Th1} and \ref{Th2} were
suggested by reading Torsten Sillke's paper \cite{Sillke}.
Thanks to Hrothgar for sending a copy of \cite{Hroth}.
Figures \ref{FigFred1}--\ref{Fig780}
were produced with the help of the CellularAutomaton command 
in {\em Mathenatica} \cite{MMA}. 
Kellen Myers showed me how to make an
animated gif with {\em Mathematica}.
Thanks  to  Roman Pearce and Michael Monagan
for computing the initial terms of sequence \eqn{EqMM}.
Stephen Wolfram, Todd Rowland, and Hrothgar provided
helpful comments on the manuscript.


\bigskip
\hrule
\bigskip

\noindent 2010 {\it Mathematics Subject Classification}:
Primary 11B85, 37B15.

\noindent \emph{Keywords: } Automata sequences, cellular automata, 
 Moore neighborhood,  von Neumann neighborhood,
odd-rule cellular automata,  run length transform, 
Fredkin Replicator, Rule 110, Rule 150


\begin{thebibliography}{99}

\bibitem{Tagaki11}
P. C. Allaart and K. Kawamura, The Takagi function: a survey,
{\em  Real Analysis Exchange}, {\textbf  37} (2011/12),  1Ð-54;
\url{http://arxiv.org/abs/1110.1691}.

\bibitem{Tooth}
D. Applegate, O. E. Pol, and N. J. A. Sloane,
The toothpick sequence and other sequences from cellular automata, 
 {\em Congress. Numerant.}, {\textbf 206} (2010), 157--191;
 \url {http://arxiv.org/abs/1004.3036}.

\bibitem{ESZ}
S. B. Ekhad, N.~J.~A.~Sloane, and D. Zeilberger,
A meta-algorithm for creating fast algorithms for counting ON cells in odd-rule cellular automata, Preprint, March 2015.

\bibitem{3by3}
S. B. Ekhad, N.~J.~A.~Sloane, and D. Zeilberger,
``Odd-rule'' cellular automata on the square grid,  Preprint, March 2015.

\bibitem{Epp09}
D. Eppstein,
Growth and decay in Life-like cellular automata, 2009;
\url{http://arxiv.org/abs/0911.2890}.

\bibitem{FSS}
S. Finch, P. Sebah, and Z.-Q. Bai, 
Odd entries in Pascal's trinomial triangle, 2008;
\url{http://arxiv.org/abs/0802.2654}.

\bibitem{Fred1990} 
E. Fredkin,
Digital mechanics, an informational process based on reversible universal cellular automata, 
in {\em Cellular Automata, Theory and Experiment},
ed. H. Gutowitz, MIT Press, 1990, pp.~254--270. 

\bibitem{Fred2000}
E. Fredkin, 
{\em Digital Mechanics} (Working Draft), 2000;
\url{http://64.78.31.152/wp-content/uploads/2012/08/digital_mechanics_book.pdf}.

\bibitem{HH02}
H. Havermann et al.,
Entry \seqnum{A071053} in \cite{OEIS},
2002--present.

\bibitem{Hroth}
Hrothgar,
Notes on a replicating automaton,
Preprint, July 2014.

\bibitem{Kari}
J. Kari, Theory of cellular automata: a survey, 
{\em Theoret. Comput. Sci.},
 {\textbf 334} (2005),  3-Ð33.
 
 \bibitem{Lag2011}
 J. C. Lagarias, The Takagi function and its properties,
 in {\em Functions in Number Theory and Their Probabilistic Aspects}, 
 ed. K.~Matsumoto et al.,
 RIMS Lecture Notes, vol. {\textbf B34}, 
 Res. Inst. Math. Sci., Kyoto, 2012, pp.~153--189;
 \url{http://arxiv.org/abs/1112.4205}.
 
 \bibitem{Layman}
 J. Layman et al., Entry \seqnum{A160239} in \cite{OEIS},
 2009--present.
 
 \bibitem{MOW84}
 O. Martin, A. M. Odlyzko, and S. Wolfram, 
 Algebraic properties of cellular automata, 
 {\em Comm. Math. Phys.}, {\textbf 93} (1984), 219--258.

\bibitem{MK06}
S. Mitra and S. Kumar,
Fractal replication in time-manipulated one-dimensional cellular automata,
{\em Complex Systems}, {\textbf 16} (2006), 191--207. 

\bibitem{OEIS}
The OEIS Foundation Inc.,
{\em The On-Line Encyclopedia of Integer Sequences}, 1996--present;
\url{https://oeis.org}.

\bibitem{PaWo85}
N. H. Packard and S. Wolfram,
Two-dimensional cellular automata,
{\em  J.  Statist. Phys.}, {\textbf 38} (1985), 901--946.

\bibitem{GFUN}
B. Salvy and P. Zimmermann,
GFUN: a Maple package for the manipulation of generating and holonomic functions in one variable,
{\em ACM Transactions on Mathematical Software}, 
{\textbf 20} (1994), 163--177.

\bibitem{Sillke}
T. Sillke,
Odd trinomials: $t(n) = (1+x+x^2)^n$, 2004;
\url{http://www.mathematik.uni-bielefeld.de/~sillke/PUZZLES/trinomials}.

\bibitem{Sing03}
D. Singmaster,
On the cellular automaton of Ulam and Warburton, 
{\em M500 Magazine of the Open University}, 
No.~{\textbf 195} (December 2003), pp. 2--7;
\url{https://oeis.org/A079314/a079314.pdf}.

\bibitem{Ulam62}
S. M. Ulam, On some mathematical problems connected with patterns of growth of figures, in
{\em Mathematical Problems in the Biological Sciences}, 
ed. R. E. Bellman, 
Proc. Sympos. Applied Math.,
Vol. {\textbf 14}, Amer. Math. Soc., 1962,
pp.~215--224.

\bibitem{Weiss30}
E. W. Weisstein, MathWorld, Entries for Rules 30, 90, 110, 150, 182, etc.; 
\url{http://mathworld.wolfram.com/}, 2004--present.


\bibitem{Wolf83}
S. Wolfram, Statistical mechanics of cellular automata, 
{\em Rev. Mod. Phys.},  {\textbf 55} (1983), 601--644.

\bibitem{Wolf84}
S. Wolfram, 
Universality and complexity in cellular automata 
({\em Cellular Automata, Los Alamos, 1983}),
{\em Physica D},
{\textbf 10} (1984, 1Ð-35. 

\bibitem{MMA}
S. Wolfram,
 {\em The Mathematica Book}, 
 Cambridge University Press and Wolfram Research, Inc.,  NY, 2000.

\bibitem{NKS}
S. Wolfram, 
{\em A New Kind of Science}, 
Wolfram Media, Champaign, IL, 2002.

\end{thebibliography}
\end{document}